\documentclass[twoside,11pt]{article}
\usepackage{graphicx,amscd,amsfonts,amsmath,amsthm,amssymb,latexsym,amsfonts,color}
\usepackage{epsfig,float,epstopdf,array,bm,cite}
\usepackage{cases,array}
\usepackage{multirow}
\usepackage[table]{xcolor}
\usepackage{tikz}
\usepackage{lmodern}
\usepackage{mathtools}
\usepackage{tikz}
\usepackage{pgfplots}
\usepackage{relsize}
\usepackage{graphicx}
\usepackage{float}
\usepackage[bookmarksnumbered, colorlinks,plainpages]{hyperref}
\def\diag{\textrm{diag}}

\newtheorem{theorem}{\bf Theorem}
\newtheorem{lemma}{\bf Lemma}

\newtheorem{example}{\bf Example}
\newtheorem{definition}{\bf Definition}
\newtheorem{remark}{\bf Remark}
\begin{document}
\title{\bf Semi-convergence of the APSS method for a class of nonsymmetric three-by-three singular saddle point problems }
\author{\small\bf  Hamed Aslani$^\dag$, Davod Khojasteh Salkuyeh$^{\ddag}$\thanks{\noindent Corresponding author. \newline 
		Emails:	hamedaslani525@gmail.com (H. Aslani), khojasteh@guilan.ac.ir (D.K. Salkuyeh)}\\[2mm]
	\textit{{\small $^\dag$Faculty of Mathematical Sciences, University of Guilan, Rasht, Iran}} \\
	\textit{{\small $^\ddag$Faculty of Mathematical Sciences, and Center of Excellence for Mathematical Modelling,}}\\
	\textit{{\small Optimization and Combinational Computing (MMOCC), University of Guilan, Rasht, Iran}}  \\
}
\date{}
\maketitle
\vspace{-0.5cm}
\noindent\hrulefill\\
{\bf Abstract.}  For nonsymmetric block three-by-three singular saddle point problems arising from the Picard iteration method for a class of mixed finite element scheme, recently  Salkuyeh et al. in (D.K. Salkuyeh, H. Aslani, Z.Z. Liang, An alternating positive semi-definite splitting preconditioner for the three-by-three block saddle point problems, Math. Commun. 26 (2021) 177-195)  established an alternating positive semi-definite splitting (APSS) method. In this work, we analyse the semi-convergence  of the APSS method for solving a class of nonsymmetric block three-by-three singular saddle point problems. The APSS induced preconditioner is applied to improve the semi-convergence  rate of the flexible GMRES (FGMRES) method. Experimental results are designated to support the theoretical results. These results show that the served preconditioner is efficient compared with FGMRES without a preconditioner. \\[-3mm]


\noindent\hrulefill\\

\pagestyle{myheadings}\markboth{H. Aslani, D.K. Salkuyeh}{Semi-convergence  analysis of the APSS method  for  3$\times$ 3 singular saddle point problems   }
\thispagestyle{empty}

\section{Introduction}

We are interested in solving  the following large and sparse block three-by-three saddle point system by iteration methods,
\begin{equation}\label{eq1s}
\left(\begin{array}{ccc}
{A} & {B^{T}} & {0} \\
{B} & {0} & {C^{T}} \\
{0} & {C} & {0}
\end{array}\right)\left(\begin{array}{l}
{x} \\
{y} \\
{z}
\end{array}\right)=\left(\begin{array}{l}
{f} \\
{g} \\
{h} 
\end{array}\right) 
\end{equation}
where $A\in \mathbb{R}^{n \times n}, B\in \mathbb{R}^{m \times n}  $ and $C\in  \mathbb{R}^{l \times m}. $ Here $ f  \in \mathbb{R}^{n }, g \in \mathbb{R}^{m}$ and $h \in \mathbb{R}^{l}$ are known and ${\bf x}=\left(x; y; z\right)$  is an unknown  vector that has to be determined. The coefficient matrix of the system \eqref{eq1}  is of order $\textbf{n}=m+n+l$. Many practical applications produce linear systems of the form \eqref{eq1}; for example  in application of  the Picard iteration method for a class of mixed finite element scheme for stationary magnetohydrodynamics models \cite{Picard} and finite element method to solve the time-dependent Maxwell equations having discontinuous coefficients  in  polyhedral domains with Lipschitz boundary. References \cite{Ap1,Ap2} provide additional instances and references therein. 

In this paper, we make an attempt to solve the following system, 
\begin{equation}\label{eq1}
\mathcal{A} {\bf x} \equiv\left(\begin{array}{ccc}
{A} & {B^{T}} & {0} \\
{-B} & {0} & {-C^{T}} \\
{0} & {C} & {0}
\end{array}\right)\left(\begin{array}{l}
{x} \\
{y} \\
{z}
\end{array}\right)=\left(\begin{array}{l}
{f} \\
{g} \\
{h} 
\end{array}\right) \equiv \mathbf{b},
\end{equation}
which is an equivalent form of \eqref{eq1s}. Note that  the coefficient matrix in \eqref{eq1s} is symmetric, however, $\mathcal{A}$  is  nonsymmetric. Nevertheless, $\mathcal{A}$ has some good properties.  For instance, $\mathcal{A}$ is positive semi-definite. It means that $\mathcal{A} + \mathcal{A}^T $ is Symmetric Positive Semi-Definite (SPSD). This property can greatly improve the performance of the GMRES method for solving the system. Some other features are available in \cite{Huang1}.

Consider the two-by-two block saddle point problem 
\begin{equation}\label{eq2}
\mathcal{\tilde{A} } {\bf \tilde{x}}\equiv\left(\begin{array}{ccc}
	\tilde{A} & \tilde{B}^{T} \\
	-\tilde{B} & 0
\end{array}\right)\left(\begin{array}{l}
	\tilde{x} \\
	\tilde{y}
\end{array}\right)=\left(\begin{array}{l}
	\tilde{f} \\
	\tilde{g}
	\end{array}\right) \equiv \mathbf{\tilde{b}}.
\end{equation}
The matrix $\tilde{A}\in \mathbb{R}^ {p\times p}$ is assumed to be a Symmetric Positive Definite (SPD), $\tilde {B} \in \mathbb{R}^{p\times q}$ with $p>q$ is a matrix that $rank(B)= r<q$ (i.e., $\tilde {B}$ is a rank-deficient matrix), $\tilde{f}\in \mathbb{R}^{p}$ and $\tilde{g}\in \mathbb{R}^{q}$. Thereby, the linear system \eqref{eq2} leads to a two-by-two singular saddle point problem.
In the context of the two-by-two singular saddle point problem, Bai \cite{Bai1} established the HSS method for singular saddle point problem and also derived some conditions for guaranteeing the semi-convergence  of the HSS method. Li et al. \cite{GHSS} generalized the HSS method for solving non-Hermitian, positive semi-definite, and singular linear systems. They studied semi-convergence  analysis of the Generalized HSS (GHSS) method. In addition, an upper bound for the semi-convergence  factor was derived. In \cite{GPHSS} a Generalized Preconditioned Hermitian and skew-Hermitian splitting method (GPHSS) was considered to solve the singular saddle point problems. The semi-convergence  of the GPHSS scheme was proved under some conditions. In addition, the authors obtained the induced preconditioner and discussed the eigenvalues of the preconditioned matrix. Then, the Local Hermitian and skew-Hermitian (LHSS) method and the Modified LHSS (MLHSS) method were established in \cite{Huang1}. They  also gave the semi-convergence  conditions. Motivated by the Uzawa method, Chao et al. \cite{SOR} designed the Uzawa-SOR method for singular saddle point problems. Generalization of the Uzawa-SOR method was investigated in \cite{AOR}. The authors established the Uzawa-AOR scheme to solve \eqref{eq2}. The distribution of the eigenvalues of the iteration matrix and the semi-convergence  properties were given.  Numerical results indicate that the Uzawa-AOR method outperforms the Uzawa method \cite{Uzawa1}, the parameterized Uzawa method \cite{Bai11}, and the Uzawa-SOR method \cite{SOR}. Some other efficient methods  for the singular saddle point problem of the form \eqref{eq2}  were studied in \cite{Ma, Zhang,Zhang2}. 

Note that some partitioning techniques can be employed to represent the three-by-three coefficient matrix in \eqref{eq1} as:
$$\left(\begin{array}{cccc}
	A & B^{T} & \vdots & 0 \\
	B & 0 & \vdots & C^{T} \\
	\cdots & \cdots & \vdots & \cdots \\
	0 & C & \vdots & 0
\end{array}\right) \qquad \text{or} \qquad \left(\begin{array}{cccc}
A & \vdots  & B^{T} & 0 \\
\cdots & \vdots & \cdots & \cdots \\
B & \vdots & 0 & C^{T} \\
0 & \vdots& C & 0
\end{array}\right),$$
which are the standard two-by-two  saddle point problem. 

Many available preconditioning schemes for \eqref{eq2}, can not be  implemented for solving \eqref{eq1s}. This is because of  different properties in two problems.
Note that the matrix $\mathcal{A}$ described in \eqref{eq1} is nonsingular if $A$ is SPD and the matrices $B$ and $C$ are of full row rank
(see \cite{nonsingular1,nonsingular2,nonsingular3}).  
In recent years, many researchers have considered the three-by-three saddle point problem. Huang et al. \cite{Huang1} arranged a block diagonal preconditioner for solving the nonsingular system of the form \eqref{eq1}. The exact and inexact versions of the preconditioner were also studied.  Then, in \cite{Cao} the shift splitting (SS) and the relaxed shift splitting (RSS) method were designated.  Xie et al.  \cite{CAMWA} considered three efficient preconditioners. The authors analyzed the eigenvalues of the corresponding preconditioned matrices. Aslani et al. \cite{Aslani} presented a new method for solving \eqref{eq1} when $A$ is SPD and $B,C$ are full row rank matrices. Convergence properties of the method was derived.  Moreover, the spectral properties of the preconditioned matrix were discussed. Abdolmaleki et al. \cite{Abdolmaleki} from another perspective organized a block three-by-three diagonal preconditioner for \eqref{eq1}. A suitable estimation strategy for lower and upper bounds of eigenvalues of the preconditioned matrix was considered. In \cite{A1} an exact parameterized block SPD preconditioner and its inexact version for a class of block three-by-three saddle point problems was proposed. The authors also estimated the eigenvalue bounds for the preconditioned matrix.

In \cite{Liang}, Liang and Zhang proposed the Alternating Positive Semi-definite Splitting (APSS) iteration method for double saddle point problems. 
Using the idea of \cite{Liang}, Salkuyeh et al. in \cite{Salkuyeh} applied the APSS method for solving the problem \eqref{eq1} and proved its convergence. 
 In the case of $C$ being rank-deficient, the coefficient matrix \eqref{eq1} is singular. Accordingly, the linear system \eqref{eq1} is labeled as a singular three-by-three saddle point problem. In this work, the three-by-three large, sparse, and singular saddle point problem is considered and the  semi-convergence  analysis of the APSS method is discussed.

The structure of this paper is as follows. The paper is started with a review of the APSS method and its corresponding induced preconditioner. In Section \ref{sec2}, we focus on the semi-convergence  properties of the APSS method for solving \eqref{eq1}. Unconditionally semi-convergent for the APSS iteration method are derived in Section \ref{Sec3}. A strategy is given to estimate the parameter of the proposed method in Section \ref{Sec4}.   Section \ref{Sec5} is devoted to give some numerical tests  to support the theoretical results. The end of the paper will be companied by some brief conclusions.    

\section{Review of the APSS method}\label{sec2}

Let us first give a brief study of the APSS method. Consider the following decomposition  for the coefficient matrix $\mathcal{A}$  in \eqref{eq1}: 
\begin{equation}\label{Split}
\mathcal{A}=\mathcal{A}_{1}+\mathcal{A}_{2},
\end{equation}
where 
\begin{equation}\label{eq111}
\mathcal{A}_{1}=\left(\begin{array}{ccc}
{A} & {B^{T}} & {0} \\
{-B} & {0} & {0} \\
{0} & {0} & {0}
\end{array}\right), \quad
\mathcal{A}_{2}=\left(\begin{array}{ccc}
{0} & {0} & {0} \\
{0} & {0} & {- C^{T}} \\
{0} & {C} & {0}
\end{array}\right).
\end{equation}
Let $\alpha> 0$ be a given parameter. Based on the decomposition \eqref{Split},
the following splittings for the matrix $\mathcal{A}$  can be stated
$$\mathcal{A}=(\alpha \mathcal{I} + \mathcal{A}_{1} ) - (\alpha \mathcal{I} - \mathcal{A}_{2} )=(\alpha \mathcal{I} + \mathcal{A}_{2} )- (\alpha \mathcal{I} - \mathcal{A}_{1} ),$$
where $\mathcal{I}$ is the identity matrix of order $\mathbf{n}$. Now, by using these splittings the APSS method can be written as
\begin{numcases}{}
(\alpha \mathcal{I}+\mathcal{A}_{1}) x^{\left(k+\frac{1}{2}\right)}=(\alpha \mathcal{I}- \mathcal{A}_{2}) x^{(k)}+b, \notag \\
(\alpha \mathcal{I}+\mathcal{A}_{2}) x^{(k+1)}=(\alpha \mathcal{I}- \mathcal{ A}_{1}) x^{\left(k+\frac{1}{2}\right)}+ b, \notag
\end{numcases}
where $x^{(0)} \in \mathbb{R}^{\mathbf{n}}$ is an initial guess.
 By eliminating $x^{(k+\frac{1}{2})},$  the iteration scheme can be rewritten as the stationary form
\begin{equation}\label{eq3}
x^{k+1}= \mathcal{T_{\alpha}} x^{k} + f,
\end{equation}
with 
\begin{equation}\label{eqt}
\mathcal{T_{\alpha}}=(\alpha \mathcal{I} + \mathcal{A}_{2})^{-1} (\alpha \mathcal{I} - \mathcal{A}_{1}) (\alpha \mathcal{I} + \mathcal{A}_{1})^{-1} (\alpha \mathcal{I} - \mathcal{A}_{2}),
\end{equation}
and
$$f=2 \alpha (\alpha \mathcal{I}+\mathcal{A}_{2})^{-1} (\alpha \mathcal{I}+\mathcal{A}_{1})^{-1} b .$$ 
Similar to the Hermitian and Skew-Hermitian splitting (HSS) iteration method \cite{HSS}, if we set 
$$\tilde{\mathcal{M}_{\alpha}}=\frac{1}{2 \alpha} (\alpha \mathcal{I}+\mathcal{A}_{1}) (\alpha \mathcal{I}+\mathcal{A}_{2}), \quad \tilde{\mathcal{N}_{\alpha}} = \frac{1}{2 \alpha} (\alpha \mathcal{I}-\mathcal{A}_{1}) (\alpha \mathcal{I}-\mathcal{A}_{2}),$$
then $\mathcal{A}=\tilde{\mathcal{M}_{\alpha}} - \tilde{\mathcal{N}_{\alpha}}$ and
$$\mathcal{T}_{\alpha}=\tilde{\mathcal{M}_{\alpha}}^{-1} \tilde{\mathcal{N}_{\alpha}}=\mathcal{I}-\tilde{\mathcal{M}_{\alpha}}^{-1} {\mathcal{A}}.$$
From now on, we use $\mathcal{M}_{\alpha}= (\alpha \mathcal{I}+\mathcal{A}_{1}) (\alpha \mathcal{I}+\mathcal{A}_{2})$ as the APSS preconditioner, since the pe-factor $\frac{1}{2\alpha}$ has no effect on the preconditioned matrix. So, the saddle point system \eqref{eq1} can be preconditioned  from the left as ${\mathcal{M}_{\alpha}}^{-1} \mathcal{A} {\bf x}= {\mathcal{M}_{\alpha}}^{-1} {\bf b} $. In this case, we have
\begin{equation}\label{eq2.5}
 {\mathcal{M}_{\alpha}}^{-1} \mathcal{A} {\bf x}=(\mathcal{I}-{\mathcal{M}_{\alpha}}^{-1} {\mathcal{N}_{\alpha}}) {\bf x}=  {\mathcal{M}_{\alpha}}^{-1} {\bf b}, 
\end{equation}
where 
\begin{equation}\label{Nalpha}
{\mathcal{N}_{\alpha}}=(\alpha \mathcal{I}-\mathcal{A}_{1}) (\alpha \mathcal{I}-\mathcal{A}_{2})=\alpha \mathcal{I} . \alpha \mathcal{I} -\alpha (\mathcal{A}_{1}+\mathcal{A}_{2})+\mathcal{A}_{1} \mathcal{A}_{2}.
\end{equation}

\section{The semi-convergence  of APSS iteration method}\label{Sec3}

In this section, we will analyze the semi-convergence  properties of the APSS iteration method for solving the double saddle point problem \eqref{eq1}. First, let us give some related definitions, lemmas and theorems.

\begin{definition}
	The iteration method \eqref{eq3} is said to be  semiconvergent if for any initial guess ${\bf{x}}^{(0)}$, the iteration sequence $\{\bf{x}^{(k)}\}$ produced by \eqref{eq3} converges to a solution  of linear systems $\mathcal{A} \bf{x}=\bf{b}.$ 
\end{definition}
\begin{theorem}
	\cite{Semi} The iteration method  \eqref{eq3} is semi-convergent if and only if 
	 $$index (\mathcal{I}-\mathcal{T}_{\alpha})=1 \quad \text{and} \quad \vartheta(\mathcal{T}_{\alpha})<1,$$
	  where  $$\vartheta(\mathcal{T}_{\alpha}) \equiv \max \{|\lambda|: \lambda \in \sigma(\mathcal{T}_{\alpha}), \lambda \neq 1\}<1,$$
	in which  $\vartheta(\mathcal{T}_{\alpha})$  is called the pseudo-spectral radius of $\mathcal{T}_{\alpha}.$
	Moreover, the limit of the sequence can be written as
	$${\bf{x}}^{(*)}= (\mathcal{I}-\mathcal{T}_{\alpha})^{D} f+[\mathcal{I}-(\mathcal{I}-\mathcal{T}_{\alpha})^{D} (\mathcal{I}-\mathcal{T}_{\alpha})] {\bf{x}}^{(0)},$$
	where $(\mathcal{I}-{\mathcal{T}_{\alpha}})^{D}$ denotes the Drazin inverse  of $(\mathcal{I}-{\mathcal{T}_{\alpha}})$ \cite{Drazin}.
	\end{theorem}

\begin{lemma}\label{lem1}
\cite{Marchuk} (Kellogg's lemma)
 If $A \in \mathbb{C} ^{n \times n}$ is positive semi-definite, then 
$$\left\|(\alpha I+A)^{-1}(\alpha I-A)\right\|_{2} \leq 1,$$
for all $\alpha >0$. Moreover, if $A \in \mathbb{C} ^{n \times n}$ is positive definite, then 
$$\left\|(\alpha I+A)^{-1}(\alpha I-A)\right\|_{2} <1,$$
for all $\alpha >0.$
\end{lemma}
\begin{remark}
	For the saddle point matrix of the form 
	\begin{equation*}
	\mathcal{A}=\left(\begin{array}{ccc}
	{A} & {B^{T}} & {0} \\
	{-B} & {0} & {- C^{T}} \\
	{0} & {C} & {0}
	\end{array}\right),
	\end{equation*}
	if $\mathcal{A}$ is singular, we can easily see that at least one of the sets $null(B^{T}) \cap null(C)$ and $ null(C^{T})$ is nontrivial, i.e., the dimension of one of the two sets is at least one.
\end{remark}
Noticing that 
\begin{eqnarray}
(\alpha \mathcal{I} - \mathcal{A}_{1}) (\alpha \mathcal{I} + \mathcal{A}_{1})^{-1} &=& (\alpha \mathcal{I} + \mathcal{A}_{1})^{-1} (\alpha \mathcal{I} - \mathcal{A}_{1})  \label{Eqq1}\\
(\alpha \mathcal{I} - \mathcal{A}_{2}) (\alpha \mathcal{I} + \mathcal{A}_{2})^{-1} &=& (\alpha \mathcal{I} + \mathcal{A}_{2})^{-1} (\alpha \mathcal{I} - \mathcal{A}_{2}),	\label{Eqq2}
\end{eqnarray}
 it is easy to see that the matrix $\mathcal{T}_{\alpha}$ is similar to 
\begin{equation}\label{eq4}
\mathcal{L_{\alpha}}=(\alpha \mathcal{I} + \mathcal{A}_{1})^{-1} (\alpha \mathcal{I} - \mathcal{A}_{1}) (\alpha \mathcal{I} + \mathcal{A}_{2})^{-1} (\alpha \mathcal{I} - \mathcal{A}_{2}).
\end{equation}
Now, since  the matrices $\mathcal{A}_{1}$ and $\mathcal{A}_{2}$ are both positive semi-definite, then using the Kellogg's lemma we have
\[
	\|\mathcal{L_{\alpha}}\|_2 \leq \|(\alpha \mathcal{I} + \mathcal{A}_{1})^{-1} (\alpha \mathcal{I} - \mathcal{A}_{1})\|_2  \|(\alpha \mathcal{I} + \mathcal{A}_{2})^{-1} (\alpha \mathcal{I} - \mathcal{A}_{2})\|_2 \leq 1.
\]
Thus, it holds that $\mathcal{L}_{\alpha} \mathbf{x} = \mathbf{x}$ if and only if  ${}\mathcal{L}_{\alpha}^{*} \mathbf{x} = \mathbf{x},$ for any  $\mathbf{x} \in \mathbb{C}^{n \times n}$ \cite{Bai1}. As a result the index of $\text{index}(\mathcal{I} - \mathcal{L}_{\alpha})=1 $. Eventually, since two similar matrices have the same index, we see that  
\begin{equation}\label{indp}
\text{index}(\mathcal{I} - \mathcal{T}_{\alpha}) = 1.
\end{equation}

Next, we give the conditions for $\vartheta(\mathcal{T}_{\alpha}) < 1.$


\begin{theorem}\label{th2}
	Suppose that $A \in \mathbb{R}^{n \times n}$ is symmetric positive definite, $B \in \mathbb{R}^{m \times n}$ is an arbitrary matrix and $C \in \mathbb{R}^{l \times m}$ is rank-deficient. Then, for the decomposition \eqref{Split} of the matrix $\mathcal{A}$, 
	the pseudo-spectral radius of $\mathcal{T}_{\alpha}$ is less than one,  i.e., $\vartheta(\mathcal{T}_{\alpha})<1, \forall \alpha>0.$

%
%
%
%
%
%
\end{theorem}
\begin{proof}
First of all note that the matrix $\alpha \mathcal{I}-\mathcal{A}_{2}$ is non-singular, and the equalities \ref{Eqq1} and \ref{Eqq2} hold. Hence, the matrix $\mathcal{T}_{\alpha}$ is similar to
 $$\mathcal{P}_{\alpha}=(\alpha \mathcal{I} + \mathcal{A}_{2})^{-1}(\alpha \mathcal{I} - \mathcal{A}_{2}) (\alpha \mathcal{I} + \mathcal{A}_{1})^{-1} (\alpha \mathcal{I} - \mathcal{A}_{1}),$$
 So, $\vartheta(\mathcal{T}_{\alpha})=\vartheta(\mathcal{P}_{\alpha}).$
	Let $x \in \mathbb{C}^{n}$  be any eigenvector of the matrix $\mathcal{P}_{\alpha}$ and $\lambda$ be the eigenvalue
	of matrix $\mathcal{P}_{\alpha}$ corresponding to eigenvector $x$, i.e., $\mathcal{P}_{\alpha} x = \lambda x.$  Without loss of generality, we assume that $\|x\|_{2}=1.$ 	Now, we prove the theorem according to the following four cases:
	
	\bigskip
	
	\noindent \textbf{\textit{Case 1.}}  $x \in null(\mathcal{A}_{1}) \cap null(\mathcal{A}_{2}).$ So, we have $\mathcal{A}_{1} x=\mathcal{A}_{2}x=0$ and it follows that 
\begin{equation}
(\alpha \mathcal{I}+ \mathcal{A}_{1})x=(\alpha \mathcal{I}-\mathcal{A}_{1})x\Longrightarrow  
x=(\alpha \mathcal{I}+ \mathcal{A}_{1})^{-1}(\alpha \mathcal{I}-\mathcal{A}_{1})x, \label{t1}
\end{equation}
and
\begin{equation}
(\alpha \mathcal{I}+ \mathcal{A}_{2})x=(\alpha \mathcal{I}-\mathcal{A}_{2})x \Longrightarrow 
 x=(\alpha \mathcal{I}+ \mathcal{A}_{2})^{-1}(\alpha \mathcal{I}-\mathcal{A}_{2})x.\label{t2}
\end{equation}
By combining \eqref{t1} and \eqref{t2}, we have $\mathcal{P_{\alpha}}x=x.$ Thus, $\lambda=1.$ \\

%
%
%
	
\noindent \textbf{\textit{Case 2.}}  $x \in null(\mathcal{A}_{2}),$ but $x \notin null(\mathcal{A}_{1}).$ It means that $\mathcal{A}_{1} x \neq 0$ and $\mathcal{A}_{2} x=0.$ From $\mathcal{P}_{\alpha} x = \lambda x,$ by easy computations we can obtain
\begin{equation}\label{eq7}
\lambda x=  (\alpha \mathcal{I}+\mathcal{A}_{1})^{-1} (\alpha \mathcal{I}-\mathcal{A}_{1}) x  .
\end{equation}  
Since $\mathcal{A}_{1}$ is positive semi-definite, it follows from the Kellogg's
lemma that
$$\|(\alpha \mathcal{I}+\mathcal{A}_{1})^{-1}(\alpha \mathcal{I}- \mathcal{A}_{1}\|_{2} \leq 1.$$ Therefore $|\lambda| \leq 1.$ In the following, we further prove that $|\lambda| < 1$ for any $\alpha >0.$ We will argue it by contradiction. If $|\lambda| = 1 ,$ then there  exists $ \theta \in(-\pi,\pi]$  such that 
\begin{equation}\label{eq70}
 (\alpha \mathcal{I}-\mathcal{A}_{1}) (\alpha \mathcal{I}+\mathcal{A}_{1})^{-1} x=e^{\imath \theta} x.
\end{equation}
If we set $\mathcal{V}_{\alpha}= (\alpha \mathcal{I}-\mathcal{A}_{1}) (\alpha \mathcal{I}+\mathcal{A}_{1})^{-1},$ then we can rewritten \eqref{eq70} as follows 
$$\mathcal{V}_{\alpha} x=e^{\imath \theta} x.$$
 It then follows that  
\begin{equation}\label{eq8}
\|\mathcal{V}_{\alpha} x\|_{2}=\|x\|_{2}.
\end{equation}
Now,  letting
$$w:=(u;v;p)=(\alpha \mathcal{I}+\mathcal{A}_{1})^{-1} x, \quad \text { with } \quad u \in \mathbb{C}^{n}, \quad   v \in \mathbb{C}^{m}, 
\text { and } \quad p \in \mathbb{C}^{l},$$ 
we get
\begin{equation}\label{eq71}
\|(\alpha \mathcal{I}-\mathcal{A}_{1}) w\|_{2}=\|(\alpha \mathcal{I}+\mathcal{A}_{1}) w\|_{2}.
\end{equation}
From \eqref{eq71}, it holds that
$$w^{*} (\mathcal{A}_{1}+\mathcal{A}_{1}^{*}) w = 0,$$
or equivalently,
\begin{equation*}
\left(\begin{array}{ccc}
u^{*} & v ^{*} &  p^{*}
\end{array}\right) \left(\begin{array}{ccc}
{A} & {0} & {0} \\
{0} & {0} & {0} \\
{0} & {0} & {0}
\end{array}\right) \left(\begin{array}{l}
u \\
v \\
p
\end{array}\right)=0,
\end{equation*}
which leads $u=0$, due to the symmetric positive definiteness of $A.$ On one hand, from \eqref{eq70} we have 
\begin{equation*}
(\alpha \mathcal{I}-\mathcal{A}_{1}) w=e^{\imath \theta} (\alpha \mathcal{I}+\mathcal{A}_{1}) w,
\end{equation*}
which gives
\begin{equation*}
\left(\begin{array}{ccc}
\alpha I-A & -B^{T} & {0} \\
B & \alpha I & {0} \\
{0} & {0} & \alpha I \\
\end{array}\right)\left(\begin{array}{c}
0 \\
v \\
p
\end{array}\right)=e^{\imath \theta} \left(\begin{array}{ccc}
\alpha I+A & B^{T} & {0} \\
-B & \alpha I & {0} \\
{0} & {0} & \alpha I \\
\end{array}\right)\left(\begin{array}{l}
0 \\
v \\
p
\end{array}\right),
\end{equation*}
and so 
\begin{equation*}
\left(\begin{array}{c}
-B^{T} v \\
\alpha v \\
\alpha p
\end{array}\right)=e^{\imath \theta} \left(\begin{array}{c}
B^{T} v \\
\alpha v \\
\alpha p
\end{array}\right).
\end{equation*}
It leads to
\begin{numcases}{}
B^{T} v = - e^{\imath \theta} B^{T} v, \label{eq73} \\
\alpha v = e^{\imath \theta } \alpha v , \label{eq74} \\ 
 \alpha p = e^{\imath \theta }  \alpha p. \label{eq75}
\end{numcases}
Here, if $B^{T} v = 0,$ then we  see that $v = 0$, because $B$ is a full row rank matrix. It is immediate to conclude that $\mathcal{A}_{1} x = 0,$ which is a contradiction with the assumption $x \notin null(\mathcal{A}_{1}).$ Thereby, from  \eqref{eq73} it holds $e^{\imath \theta}=-1.$ Substituting the equality $e^{\imath \theta}=-1$ into \eqref{eq74} and \eqref{eq75}, results in $v=0$ and  $p=0,$ respectively. So $w=0.$ Consequently, we have $x=(\alpha \mathcal{I}+\mathcal{A}_{1})w=0,$ which contradicts
with the fact that $x$ is an eigenvector. Therefore, $|\lambda| < 1 .$ 

\bigskip

\noindent \textbf{\textit{Case 3.}}  $x \in null(\mathcal{A}_{1}),$ but $x \notin null(\mathcal{A}_{2}).$ 
So, we have 
\[
\mathcal{A}_{1}x=0 \Longrightarrow (\alpha \mathcal{I}+\mathcal{A}_{1})x=\alpha x  \Longrightarrow 
(\alpha \mathcal{I}+\mathcal{A}_{1})^{-1}x=\frac{1}{\alpha} x.
\] 
Hence,  using  $\mathcal{A}_{1}x=0 $ and the above equation, we get 
\begin{eqnarray*}
\mathcal{P}_{\alpha} x &=& \alpha (\alpha \mathcal{I} + \mathcal{A}_{2})^{-1}(\alpha \mathcal{I} - \mathcal{A}_{2}) (\alpha \mathcal{I} + \mathcal{A}_{1})^{-1} x\\
                                        &=& (\alpha \mathcal{I}+\mathcal{A}_{2})^{-1} (\alpha \mathcal{I}-\mathcal{A}_{2})x.
\end{eqnarray*}
Therefore, using the fact that $\mathcal{P}_{\alpha}x=\lambda x$, positive semi-definiteness of $\mathcal{A}_{2}$ and the Kellogg's lemma, we deduce that
\begin{equation}\label{eq76}
|\lambda| \leq  \| \mathcal{P}_{\alpha}x\| \leq \| (\alpha \mathcal{I}+\mathcal{A}_{2})^{-1} (\alpha \mathcal{I}-\mathcal{A}_{2})\|_{2}\leq 1.
\end{equation}
 Moreover, we claim that $|\lambda| =1$ never happens. By contradiction we assume that $|\lambda|=1.$ So, there exists $\theta \in(-\pi,\pi]$ so that
\begin{equation*}
(\alpha \mathcal{I}-\mathcal{A}_{2})  x=e^{\imath \theta} (\alpha \mathcal{I}+\mathcal{A}_{2})x,
\end{equation*}
which is equivalent to
\begin{numcases}{}
\alpha x_{1} = e^{\imath \theta} \alpha x_{1}, \label{eq3a} \\
\alpha x_{2} +C^{T} x_{3}= e^{\imath \theta } (\alpha x_{2} -C^{T} x_{3}) , \label{eq3b} \\ 
-C x_{2}+\alpha x_{3} = e^{\imath \theta }  (C x_{2}+\alpha x_{3}). \label{eq3c}
\end{numcases}
 From Eq. \eqref{eq3a}, either $e^{\imath \theta}=1$ or $x_{1}=0.$ First, suppose that $e^{\imath \theta}=1.$ From $\mathcal{A}_{1}x=0,$ we have
\begin{numcases}{}
A x_{1}+B^{T}x_{2}=0, \label{n1}\\
Bx_{2}=0. \label{n2}
\end{numcases}
 Substituting  $x_{1}=-A^{-1} B^{T} x_{2}$ which is deduced from \eqref{n1} into \eqref{n2}, gives
 $B A^{-1} B^{T} x_{2}=0,$ consequently, $x_{2}=0,$ and then by \eqref{n1}, $x_{1}=0.$
 Now, from Eq. \eqref{eq3b}, we have $C^{T} x_{3}=0.$ Thus, $\mathcal{A}_{2} x=0$ which is in contradiction with the assumption. 
   Therefore, $e^{\imath \theta}=1$ never happens. Now, we discuss $x_{1}=0.$  Clearly, the assumption $\mathcal{A}_{1} x =0,$ gives 
\begin{equation*}
 \left(\begin{array}{ccc}
{A} & {B^{T}} & {0} \\
{-B} & {0} & {0} \\
{0} & {0} & {0}
\end{array}\right) \left(\begin{array}{l}
0 \\
x_{2} \\
x_{3}
\end{array}\right)=\left(\begin{array}{l}
B^{T} x_{2} \\
0 \\
0
\end{array}\right)=0.
\end{equation*}
Since $B$ has full row rank, we deduce that $x_{2}=0.$  Substituting $x_{2}=0$ into \eqref{eq3c} gives $x_{3}=0.$ So, $x=(x_{1};x_{2};x_{3})=0 ,$ which is impossible. thereupon, $|\lambda|=1$ is unacceptable.
%
\bigskip

\noindent \textbf{\textit{Case 4.}} 
  $x \notin null(\mathcal{A}_{1}),x \notin null(\mathcal{A}_{2}).$  For this case, we have $\mathcal{A}_{1} x \neq 0$ and $\mathcal{A}_{2} x \neq 0.$ 
If we define $\mathcal{Q}_{\alpha}= (\alpha \mathcal{I} + \mathcal{A}_{2} )^{-1} (\alpha \mathcal{I} - \mathcal{A}_{2} )$ and $\mathcal{V}_{\alpha}= (\alpha \mathcal{I}+ \mathcal{A}_{1} )^{-1} (\alpha \mathcal{I} - \mathcal{A}_{1} ),$ we can see that $\mathcal{P}_{\alpha}= \mathcal{Q}_{\alpha} \mathcal{V}_{\alpha}.$ Since $\mathcal{Q}_{\alpha}$ is a unitary matrix, we see that   $\left\|\mathcal{Q}_{\alpha} \right\|_{2} = 1$. On the other hand,  since $\mathcal{A}_{1}$ is positive semi-definite matrix from the Kellogg's lemma, we have $\left\|\mathcal{V}_{\alpha} \right\|_{2} \leq 1$ for any $ \alpha > 0 .$ It leads to 
\begin{equation*}
\|\mathcal{P}_{\alpha}\|_{2} \leq\|\mathcal{Q}_{\alpha}\|_{2}\|\mathcal{V}_{\alpha}\|_{2} =\|\mathcal{V}_{\alpha}\|_{2} \leq 1, \quad \forall \alpha>0.
\end{equation*}
Thus, $|\lambda|\leq 1.$ In the following, we further prove that $|\lambda| < 1,$ for any $\alpha >0.$ We will argue it by contradiction. If $|\lambda| = 1 ,$ then there  exists  $\theta \in (-\pi,\pi]$ such that 
$$\mathcal{P}_{\alpha} x=e^{\imath \theta} x,$$
which is equivalent to,

\begin{equation}\label{eq10}
\mathcal{V}_{\alpha} x=e^{\imath \theta} \mathcal{Q}_{\alpha}^{*} x.
\end{equation}
Cosquently,
\begin{equation}\label{eq8}
\|\mathcal{V}_{\alpha} x\|_{2}=\|x\|_{2}.
\end{equation}
Substituting $\mathcal{V}_{\alpha}$ into \eqref{eq8} and using the change of  variable 
$$w:=(u;v;p)=(\alpha \mathcal{I}+\mathcal{A}_{1})^{-1} x, \quad \text { with } \quad u \in \mathbb{C}^{n} \quad ,  v \in \mathbb{C}^{m} \text { and } \quad p \in \mathbb{C}^{l},$$ 
gives
\begin{equation}\label{eq9}
\|(\alpha \mathcal{I}-\mathcal{A}_{1}) w\|_{2}=\|(\alpha \mathcal{I}+\mathcal{A}_{1}) w\|_{2}.
\end{equation}
From \eqref{eq9}, it holds 
$$w^{*} (\mathcal{A}_{1}+\mathcal{A}_{1}^{*}) w = 0,$$
or equivalently,
\begin{equation*}
\left(\begin{array}{ccc}
u^{*} & v ^{*} & p^{*}
\end{array}\right) \left(\begin{array}{ccc}
{A} & {0} & {0} \\
{0} & {0} & {0} \\
{0} & {0} & {0}
\end{array}\right) \left(\begin{array}{l}
u \\
v \\
p
\end{array}\right)=0,
\end{equation*}
which leads to $u=0$, due to the symmetric positive definiteness of $A.$ On one hand, from \eqref{eq10} we have 
\begin{equation*}
(\alpha \mathcal{I}-\mathcal{A}_{1}) w=e^{\imath \theta} \mathcal{Q}_{\alpha}^{*}(\alpha \mathcal{I}+\mathcal{A}_{1}) w,
\end{equation*}
which results in
\begin{equation*}
\left(\begin{array}{ccc}
	\alpha I-A & -B^{T} & {0} \\
	B & \alpha I & {0} \\
	{0} & {0} & \alpha I \\
\end{array}\right)\left(\begin{array}{c}
	0 \\
	v \\
	p
\end{array}\right)=e^{\imath \theta} \mathcal{Q}_{\alpha}^{*}\left(\begin{array}{ccc}
	\alpha I+A & B^{T} & {0} \\
	-B & \alpha I & {0} \\
	{0} & {0} & \alpha I \\
\end{array}\right)\left(\begin{array}{l}
	0 \\
	v \\
	p
\end{array}\right).
\end{equation*}
So
\begin{equation*}
\left(\begin{array}{c}
	-B^{T} v \\
	\alpha v \\
	\alpha p
\end{array}\right)=e^{\imath \theta} \mathcal{Q}_{\alpha}^{*}\left(\begin{array}{c}
	B^{T} v \\
	\alpha v \\
	\alpha p
\end{array}\right).
\end{equation*}
The above equality can be rewritten as 
\begin{equation*}
(\alpha \mathcal{I} - \mathcal{A}_{2})\left(\begin{array}{c}
-B^{T} v \\
\alpha v \\
\alpha p
\end{array}\right)=e^{\imath \theta} (\alpha \mathcal{I} + \mathcal{A}_{2}) \left(\begin{array}{c}
B^{T} v \\
\alpha v \\
\alpha p
\end{array}\right).
\end{equation*}
 It leads to
\begin{equation*}
\left(\begin{array}{ccc}
\alpha I & {0} & {0} \\
{0} & \alpha I & C^{T} \\
{0} & -C & \alpha I \\
\end{array}\right)\left(\begin{array}{c}
-B^{T} v\\
\alpha v \\
\alpha p
\end{array}\right)=e^{\imath \theta} \left(\begin{array}{ccc}
\alpha I & {0} & {0} \\
{0} & \alpha I & -C^{T} \\
{0} & C & \alpha I \\
\end{array}\right)\left(\begin{array}{l}
B^{T}v \\
\alpha v \\
\alpha p
\end{array}\right),
\end{equation*}
that is implies that
\begin{numcases}{}
B^{T} v = - e^{\imath \theta} B^{T} v, \label{eq13} \\
\alpha v + C^{T} p = e^{\imath \theta } (\alpha v - C^{T}p), \label{eq14} \\ 
- C v+ \alpha p = e^{\imath \theta } (Cv + p). \label{eq15}
\end{numcases}
If $B^{T} v = 0,$ then we can easily see that $\mathcal{A}_{1} v = 0.$ It is immediate to conclude that $\mathcal{A}_{1} x = 0,$ which is a contradiction with the consideration $x \notin null(\mathcal{A}_{1}).$ Thereby, from  \eqref{eq13}, it holds $e^{\imath \theta}=-1.$ Substituting the identity $e^{\imath \theta}=-1$ into \eqref{eq14} and \eqref{eq15}, derives $v=0$ and  $p=0,$ respectively. So $w=0.$ Consequently, we have $x=(\alpha \mathcal{I}+\mathcal{A}_{1})w=0,$ which contradicts
with the fact that $x$ is a non-zero vector. Therefore, $|\lambda| < 1 .$  

In summary, from the  \textit{Cases  1-4}, we see that $\vartheta(\mathcal{T}_{\alpha})<1.$
\end{proof}

\begin{theorem}
	 Suppose that the assumptions of Theorem \ref{th2} hold. Then, the APSS iteration method \eqref{eq3}  is semi-convergent for any $\alpha>0$. 
	 \begin{proof}
	 	The proof immediately follows from Eq. \eqref{indp} and Theorem \ref{th2}.
	 \end{proof}
\end{theorem}

\section{Estimation strategy for the parameter $\alpha$}\label{Sec4}

Finding the optimal parameter $\alpha$ of the APSS method  is generally hard. In  this section, we propose an appropriate strategy for estimating $\alpha$ in the APSS method, which has been studied by Cao \cite{Cao2}. 

Notably, from \eqref{eq2.5} it is anticipated that $\mathcal{M}_{\alpha}$ is as close as conceivable to $ \mathcal{A}$ when $\mathcal{N}_{\alpha} \approx 0.$ In this way, having Eq. \eqref{Nalpha} in mind, the function
\begin{align*}
	\Psi(\alpha) &=\alpha\|\mathcal{I}\|_{F} \cdot \alpha\|\mathcal{I}\|_{F}-\alpha\left(\left\|\mathcal{A}_{1}\right\|_{F}+\left\|\mathcal{A}_{2}\right\|_{F}\right)+\left\|\mathcal{A}_{1}\right\|_{F}\left\|\mathcal{A}_{2}\right\|_{F} \\
	&=(n+m+l) \alpha^{2}-\alpha\left(\left\|\mathcal{A}_{1}\right\|_{F}+\left\|\mathcal{A}_{2}\right\|_{F}\right)+\left\|\mathcal{A}_{1}\right\|_{F}\left\|\mathcal{A}_{2}\right\|_{F}\\
		&={\bf{n}} \alpha^{2}-\alpha\left(\left\|\mathcal{A}_{1}\right\|_{F}+\left\|\mathcal{A}_{2}\right\|_{F}\right)+\left\|\mathcal{A}_{1}\right\|_{F}\left\|\mathcal{A}_{2}\right\|_{F},
\end{align*}
 can be characterized. Minimizing $\Psi(\alpha)$ with respect to  $\alpha$ leads to the estimation parameter $\alpha_{est}$ being set to
 $$\alpha_{es t}=\frac{\left\|\mathcal{A}_{1}\right\|_{F}+\left\|\mathcal{A}_{2}\right\|_{F}}{2{\bf{n}}}.$$
Since the matrix $A$ is SPD (in this case, $\mathcal{A}_{1}\neq 0$), one can conclude that $\alpha_{est}>0.$ In the following section, the efficiency of this choice will be verified.
\section{Numerical experiments} \label{Sec5}
 To indicate the efficiency of the APSS preconditioner, we have conducted some numerical tests. We provide two examples and each one is scaled by a symmetric diagonal matrix. To this end, the coefficient matrix $\mathcal{A}$ is replaced by the matrix $ \mathcal{D}^{-\frac{1}{2}} \mathcal{A} \mathcal{D}^{-\frac{1}{2}},$ in which ${\cal D}=\diag (\| {\mathcal{A}}_{1}\|_2,\ldots,\|{\mathcal{A}}_{\bf  n}\|_2)$. In addition, the $j$th column of the matrix $\mathcal{A}$ is represented by $\mathcal{A}_{j}.$
 
 For all the examples, a zero vector  was used  as the initial guess, and the iteration  was terminated once
$$Res=\frac{\left\|{\mathbf{b}}-\mathcal{A} {\bf{x}}^{(k)}\right\|_{2}}{\left\|{\mathbf{b}}\right\|_{2}} \leq 10^{-7},$$ 
or the specified number of iteration steps, $maxit=2000$, is exceeded. Note that ${\bf{x}}^{(k)}$ stands for the computed solution at the $k$th iteration.  The right hand side vector $\bf{b}$ was chosen such that the exact solution of \eqref{eq1} was a vector of all ones.  
 
 In the following, we apply the complete version of FGMRES method with right preconditioning technique in conjunction with APSS‌ preconditioner  $\mathcal{M}_{\alpha}.$ For the APSS preconditioner, we need to solve two SPD linear systems including $\alpha I+A+\frac{1}{\alpha} B^{T}B$ and $\alpha ^{2} I+C C^{T}$ as sub-tasks. To solve these linear systems, we employ the  conjugate gradient (CG) method without preconditioning and the iteration was stopped when the residual 2-norm  is reduced by a factor of $10^{3}$ or the maximum number of inner iteration is reached  to  200.
 
 In the tables below, we use ``CPU", ``IT" and ``RES" to represent the elapsed CPU time to converge in second, iteration counts and relative residual, respectively. In addition, Degree of freedom (DOF) is defined as $DOF=n+m+l={\bf n}$. Finally, a dagger (\dag), means that more than the maximum number of iterations are needed to converge. All the examples were performed in \textsc{Matlab-R2019a}.  All the numerical results were obtained by a  Laptop  by the following features:
 
 \hspace{2cm} intel (R), Core(TM) i5-8265U, CPU @ 1.60 GHz, 8 GB.

 \begin{example}\label{ex1} \rm
 	Consider the two dimensional leaky lid-driven cavity problem
 	\begin{equation}\label{Stokes}
 	\begin{dcases}
 	-\Delta \mathbf{u}+\nabla p=0, & \text { in } \Omega, \\
 	\nabla \cdot \mathbf{u}=0, & \text { in } \Omega,
 	\end{dcases}
 	\end{equation}	
 	in which a suitable boundary conditions are applied on the side and bottom points. $\mathbf{u}$ and $p$  symbolized the velocity vector field and the pressure scalar field, respectively. In addition, $\Delta$ and $\nabla$ refer to the vector Laplacian in $\Bbb{R}^{2}$ and  the gradient, respectively. This problem is called the Stokes problem. For the saddle point problem \eqref{eq1}, the matrices $A$ and $B$ come from the Stokes problem. To obtain these matrices, we use the IFISS package by Elman et al. \cite{IFISS}. Note that the stablized $Q_{1}-P_{0}$ and $Q_{2}-P_{1}$ finite element method (FEM) is employed to discretize the Stokes equation \eqref{Stokes}. The grid parameters are chosen as $h=\frac{1}{8}, \frac{1}{16}, \ldots, \frac{1}{256},$ for all uniform or stretched grid points. Now, the matrix $C$ is taken to be the form 
 	$$C=\left(\begin{array}{l}
 		C_{1} \\
 		C_{2}
 	\end{array}\right)=\left(\begin{array}{l}
 	C_{1} \\
 	c_{2}\\
 	c_{2}
 	\end{array}\right) \in \mathbb{R}^{(l+2)\times m}, $$
 	where $C_{1}=
 	\left(\begin{array}{ll}
 \operatorname{diag}\{1,3,5, \ldots, 2 l-1\} & \operatorname{randn}(l, m-l)
 	\end{array}\right)\in \mathbb{R}^{l\times m},$
 	and
 	$c_1=
 	\left(\begin{array}{ll}
 	e^{T}; &0^{T} \end{array}\right) C_{1},$
 	 $ c_2=
 	\left(\begin{array}{ll}
 	0^{T}; &e^{T} \end{array}\right) C_{1},$
 	$e^{T}=	\left(\begin{array}{llll}
	1; &1;&1;\cdots,&1 \end{array}\right) \in \mathbb{R}^{\frac{l}{2}}. $
 	Note that $l=m-2.$ Here, $\operatorname{randn}(l, m-l)$ is a normally distributed random matrix of order $l\times m-l.$
 	
 	This example is a technical variant of Example 1 in \cite{A1}. 
 	The results are given in Tables \ref{tab1}-\ref{tab4}. From these tables, we can observe that the FGMRES method in conjunction with the  
 	APSS preconditioner $\mathcal{M}_{\alpha}$ is strongly more efficient than the FGMRES without a preconditioner. As we can see from Figure \ref{fig1}, the preconditioner $\mathcal{M}_{\alpha}$ is efficient to cluster the eigenvalues of the original coefficient matrix. Another observation which can be posed here is that the number of  iterations of the FGMRES method with the APSS preconditioner  remains almost constant when the size of the problem is increased, whereas this is not the case for FGMRES without preconditioning.
 	 \begin{figure}[!htp]
 		\begin{center}
 			\includegraphics[height=3.75cm,width=5cm]{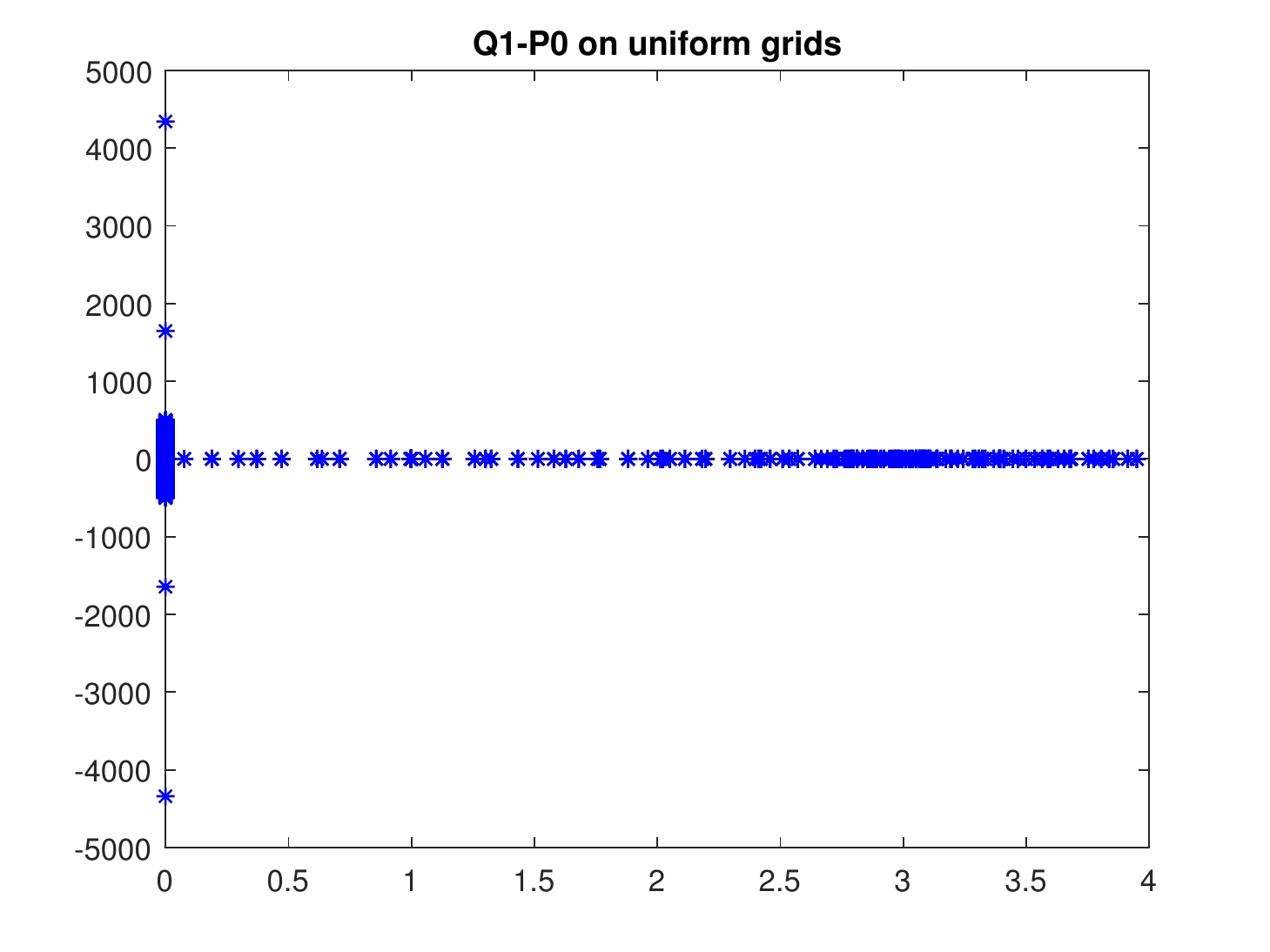}
 			\hspace*{0.2cm}		
 			\includegraphics[height=3.75cm,width=5cm]{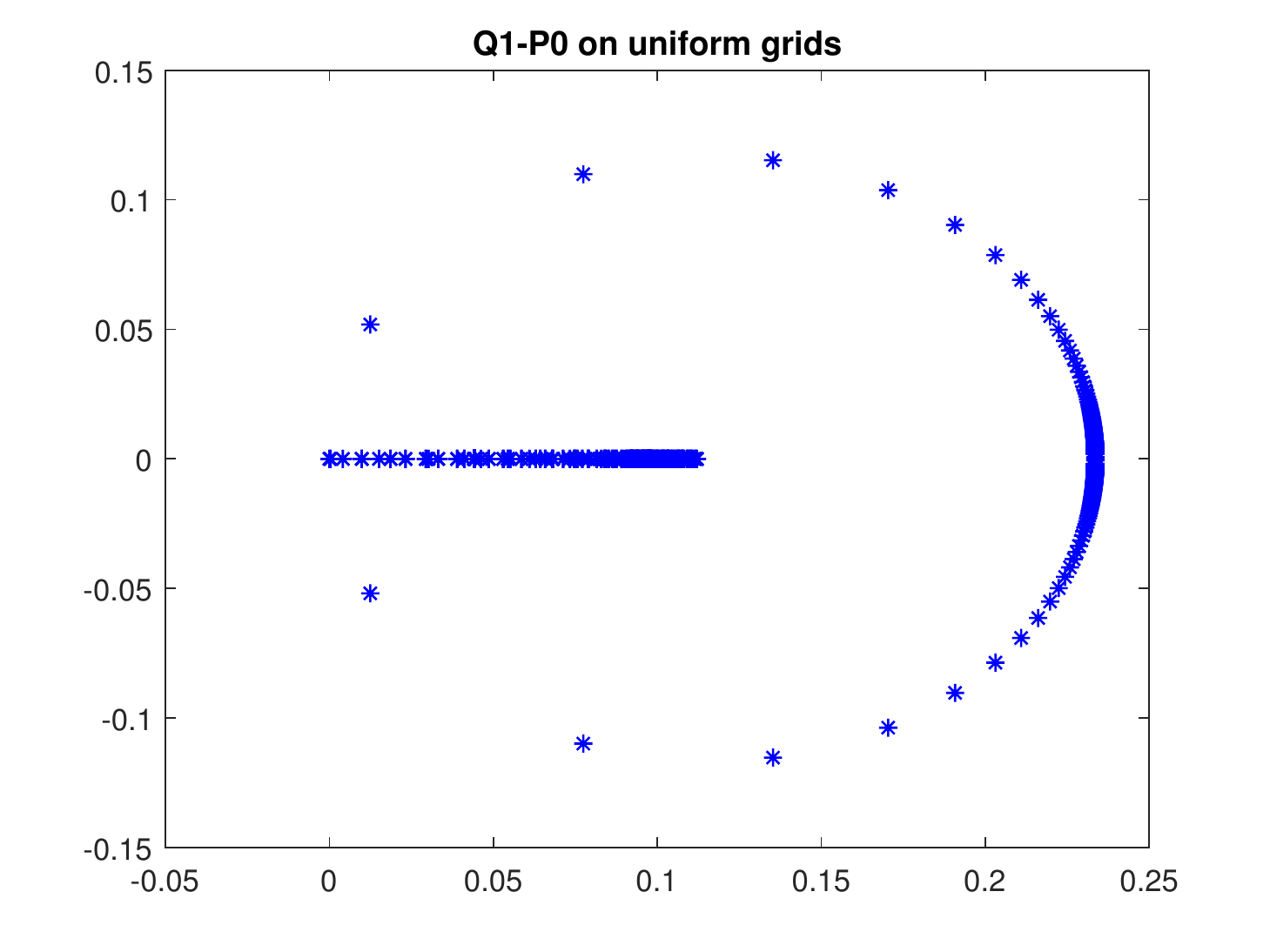}
 			\hspace*{0.2cm}
 		\includegraphics[height=3.75cm,width=5cm]{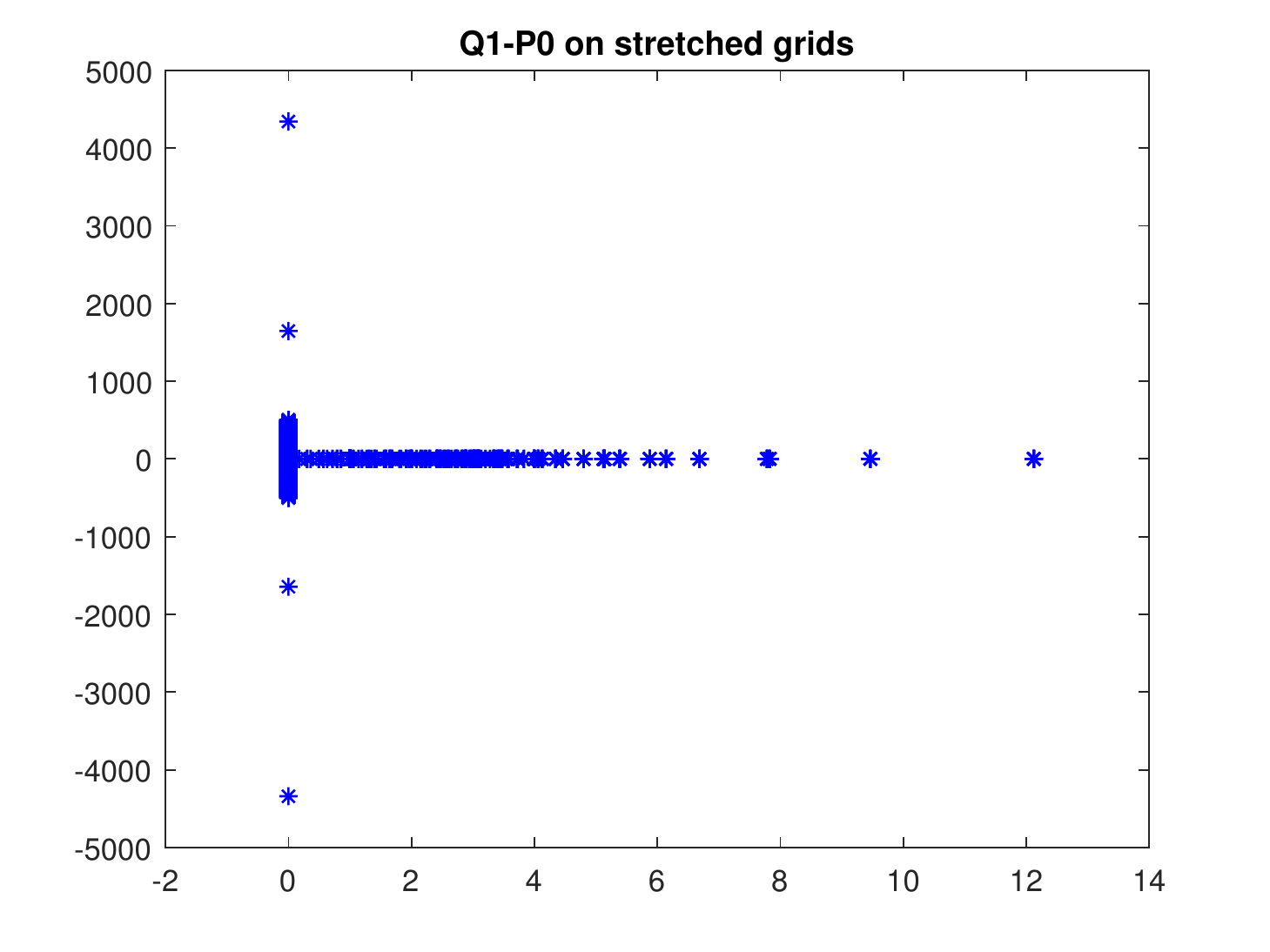}
 		\hspace*{0.2cm}	
 		\includegraphics[height=3.75cm,width=5cm]{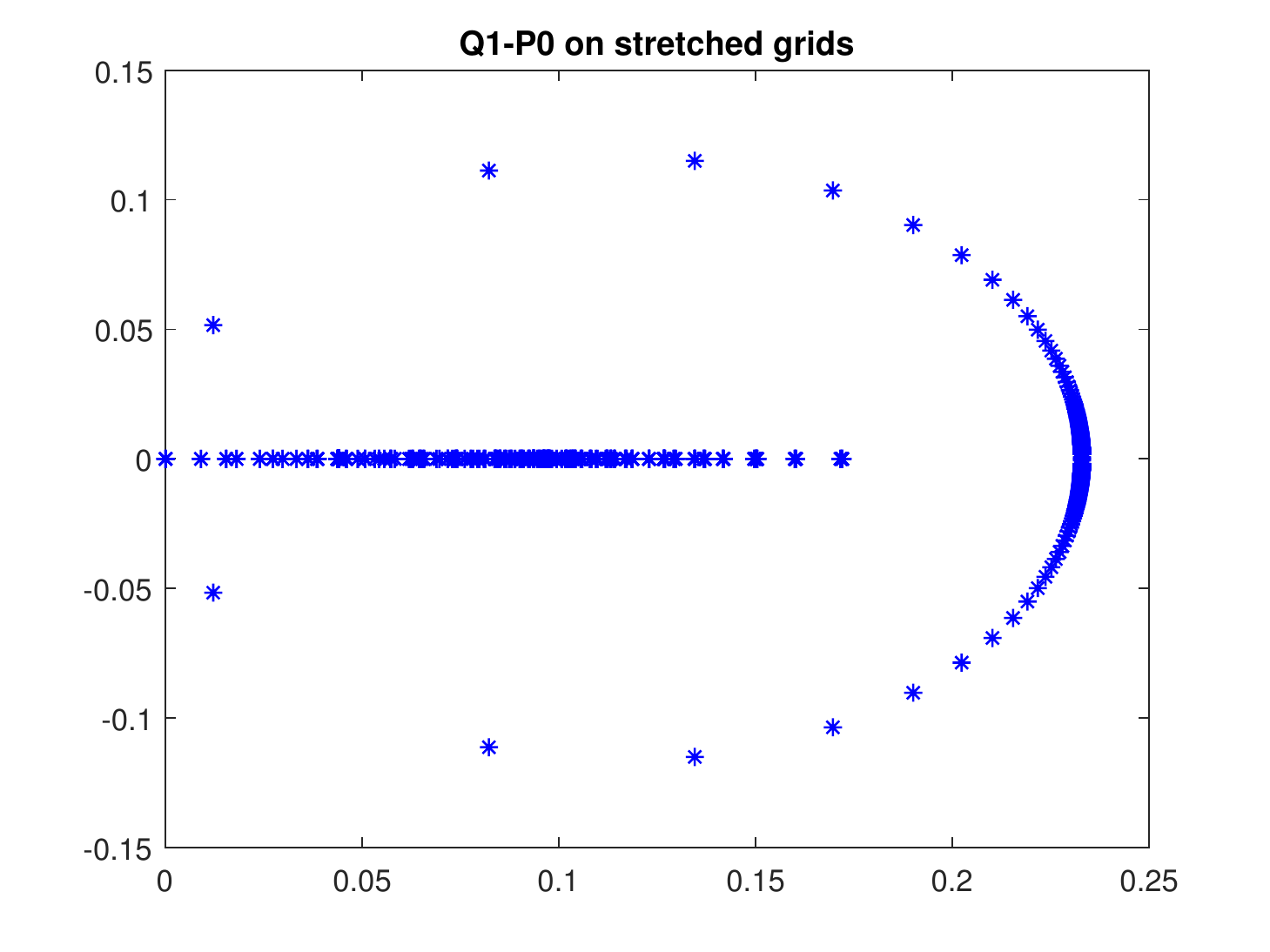}
			\includegraphics[height=3.75cm,width=5cm]{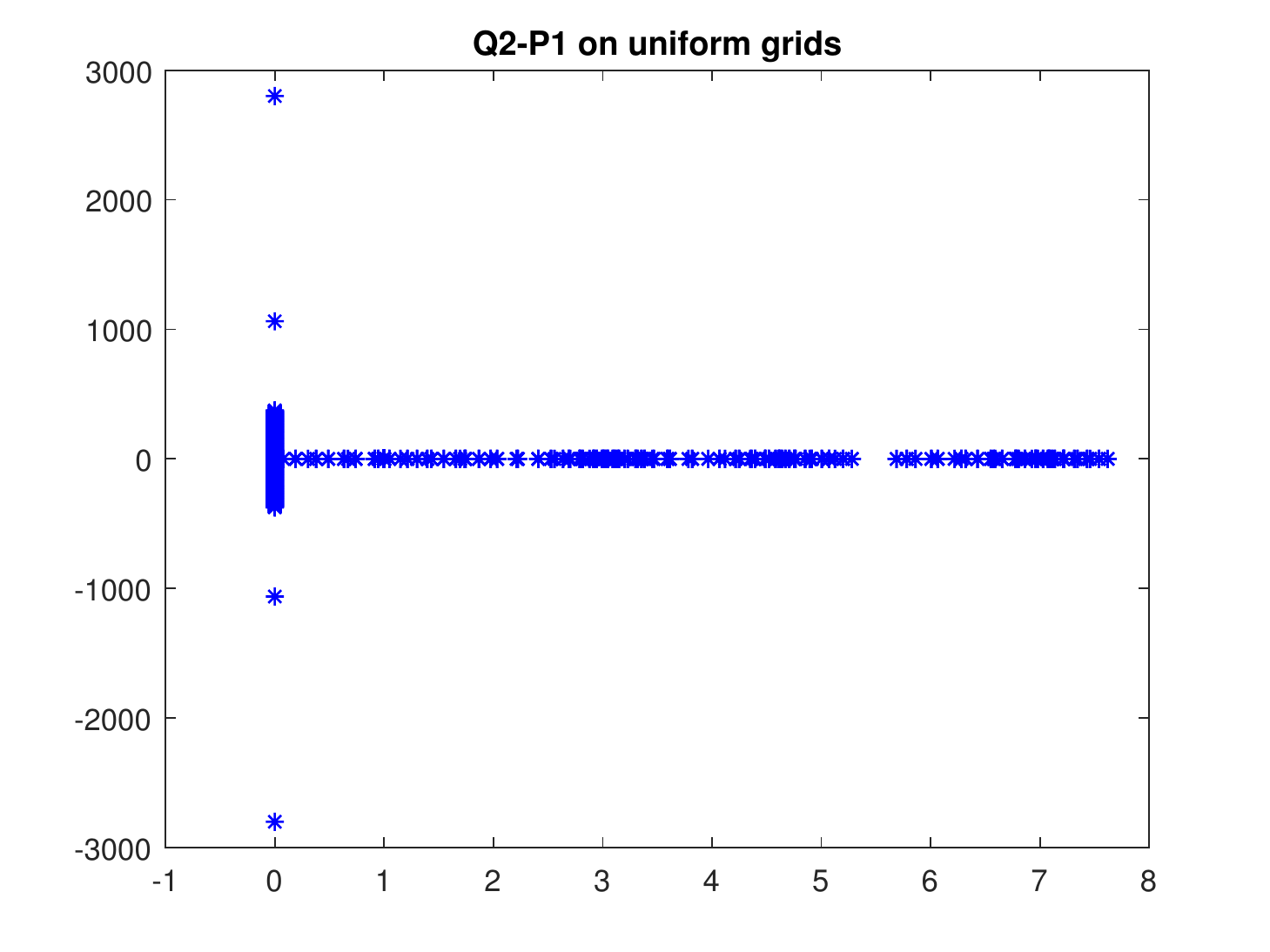}
 		\hspace*{0.2cm}
 	\includegraphics[height=3.75cm,width=5cm]{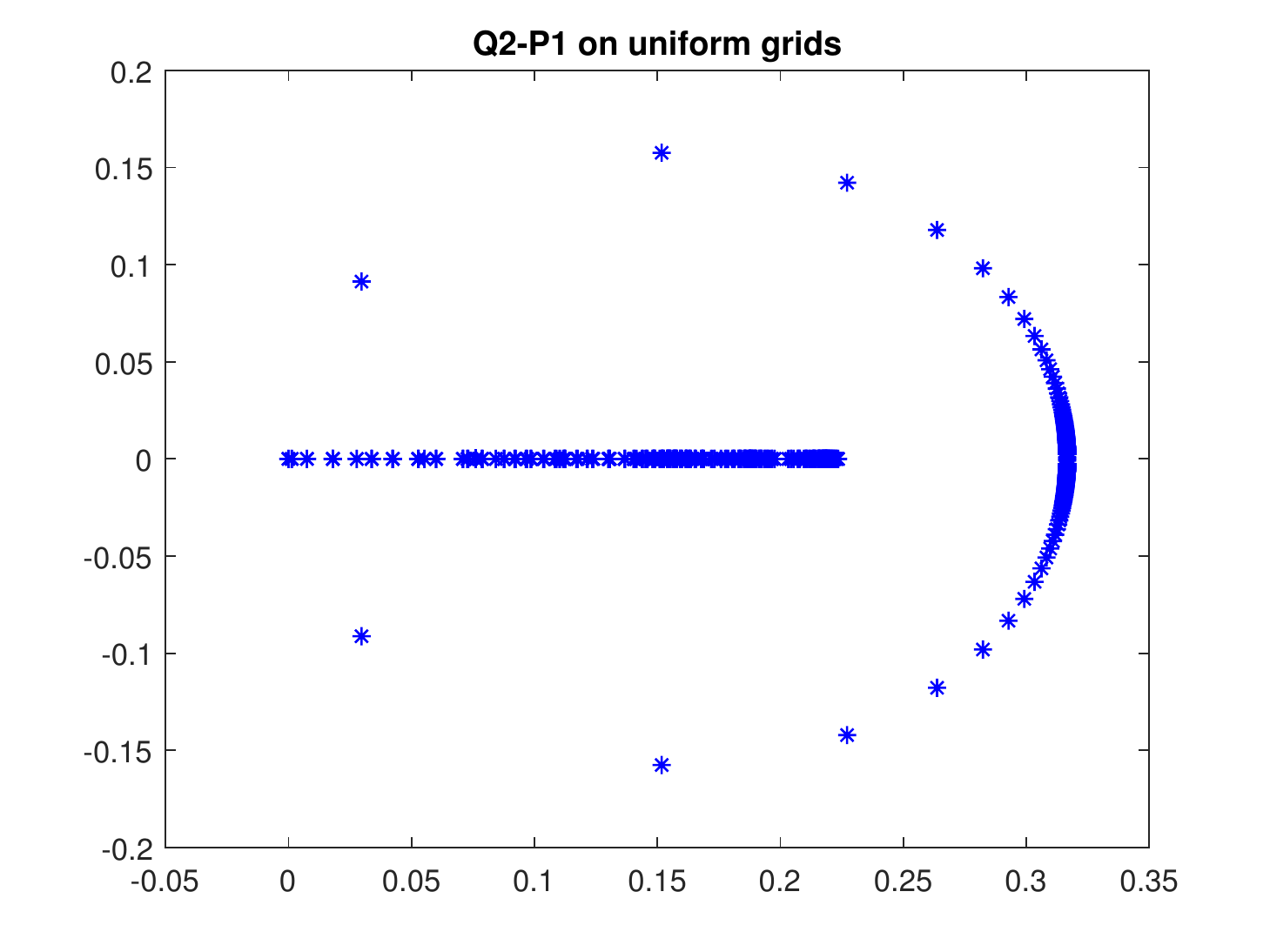}
 	\hspace*{0.2cm}
 		\includegraphics[height=3.75cm,width=5cm]{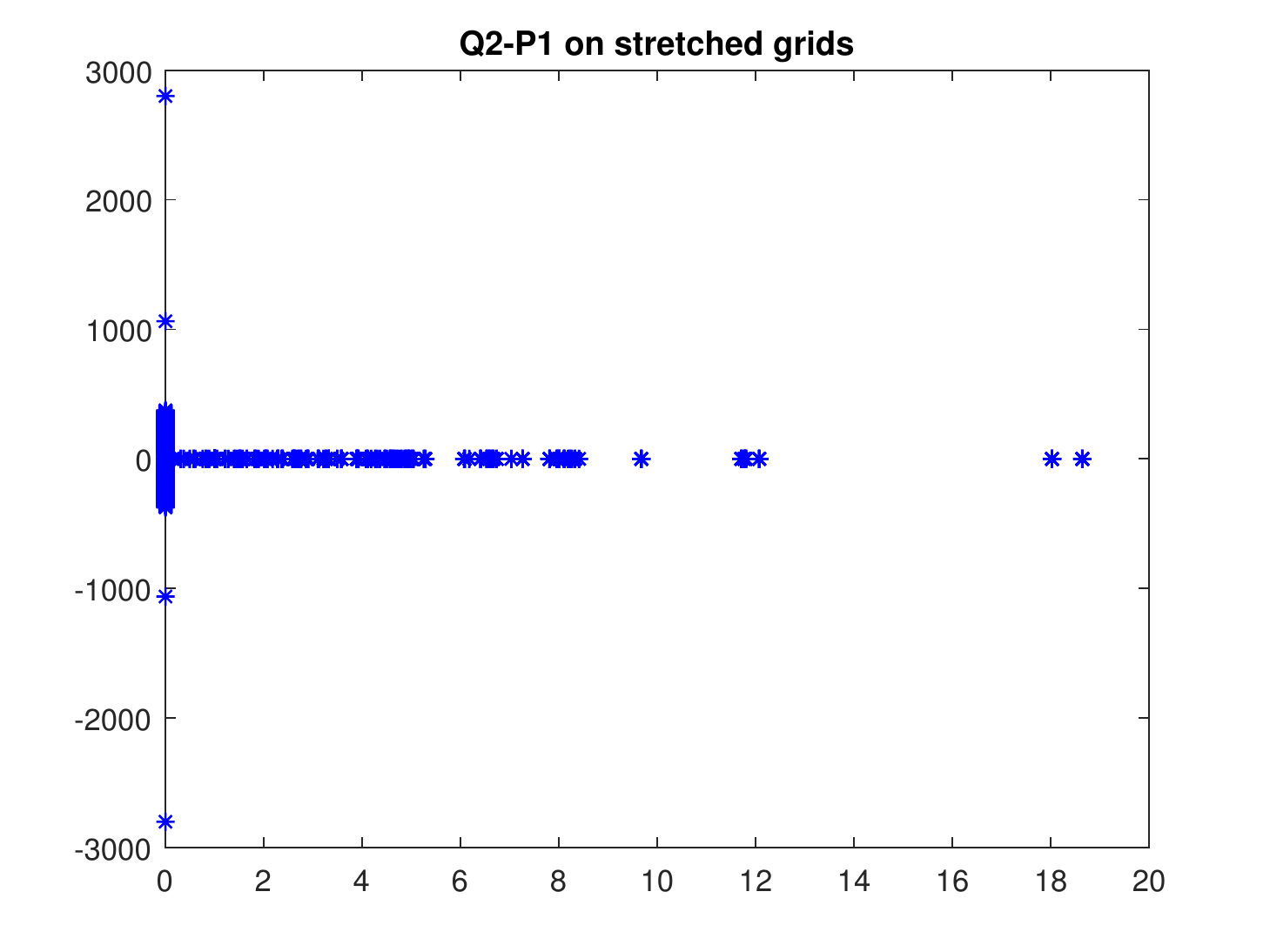}
 		\hspace*{0.2cm}		
 			\includegraphics[height=3.75cm,width=5cm]{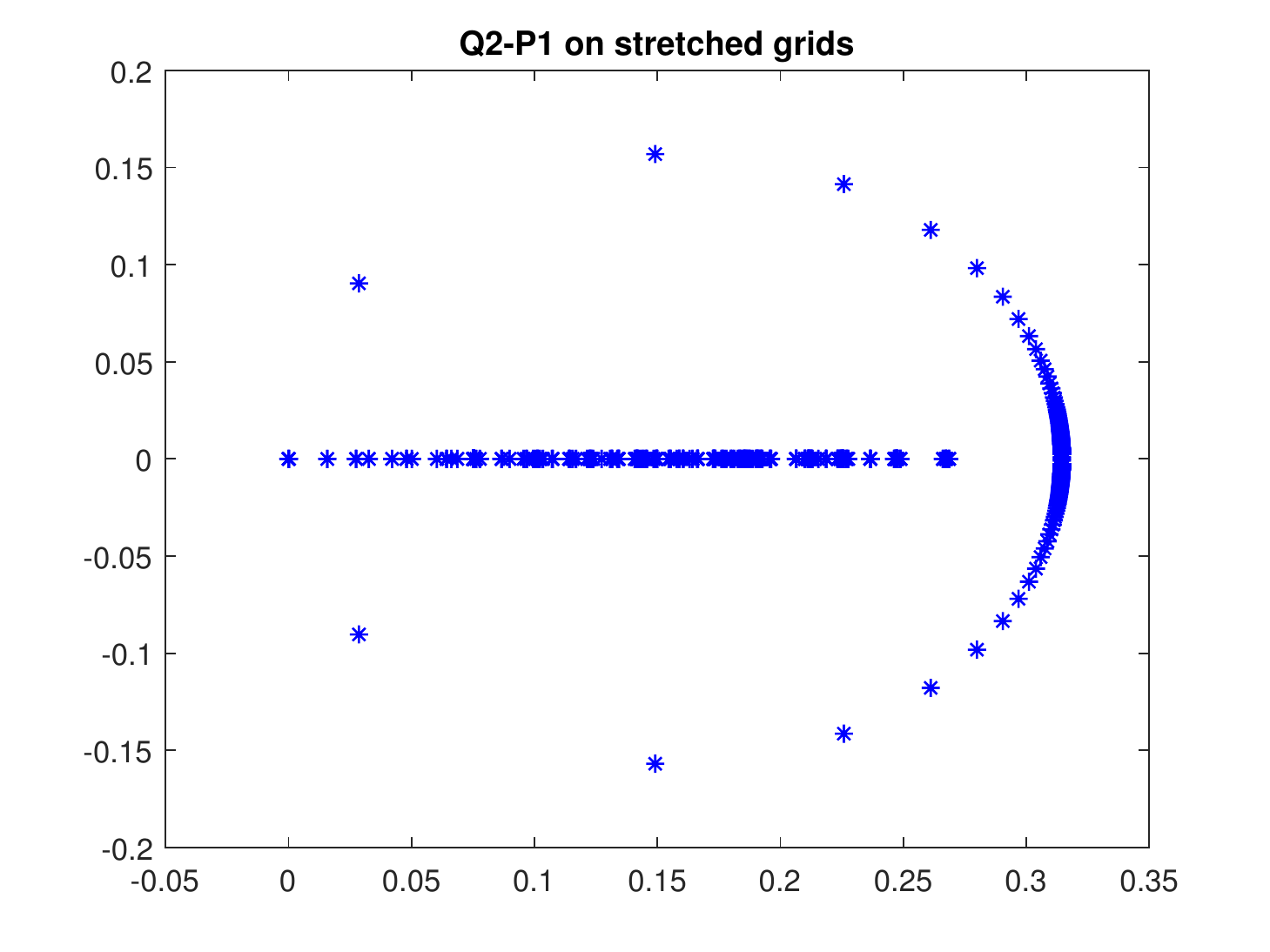}
 			
 		\end{center}
 		\caption{ Eigenvalues distribution of $\mathcal{A}$ and $\mathcal{M}_{\alpha}^{-1} \mathcal{A}$ (from the left to right) for $h=\frac{1}{16}$ for Example \ref{ex1} . \label{fig1} }	
	\end{figure}

 	\begin{table}[!t]
 		\centering
 		\caption{\small{Numerical results for FGMRES to solve Example \ref{ex1} with $Q_{1}-P_{0} $ FEM on uniform grids.}}\label{tab1}\vspace{0.25cm}
 		\begin{tabular}{|p{0.8cm}| p{1.3cm}|p{1.4cm}p{1.4cm}p{1.4cm} p{1.5cm} p{1.8cm}p{1.9cm}| }
 			\hline
 			\multirow{1}{*}{Prec.}   & $h (DOF)$            &  $\frac{1}{8}$ (286)   & $\frac{1}{16} (1086) $  &  $\frac{1}{32} (4222) $  &  $\frac{1}{64} (16638) $& $\frac{1}{128}(66046)$&$\frac{1}{256}(263166)$ \\
 			\hline\hline
 			\multirow{3}{*}{$\mathcal{I}$}&
 		     IT&49&103&552&188&474&1492\\
 			& CPU&0.0416&0.0600&0.4866&0.5810&4.1456&135.3632\\
 			& RES&2.2e-08&9.9e-08&1.0e-07&9.3e-08&9.8e-08&1.1e-07\\
 			\hline
 			\multirow{4}{*}{$\mathcal{M}_{\alpha}$}&
 		    $\alpha_{est}$&0.0396&0.0201&0.0101&0.0051&0.0025&0.0013\\
 			&IT&11&13&11&9&10&11\\
 		& CPU&0.0795&0.1009&0.0915&0.1807&0.6871&6.6438\\
 		& RES&9.9e-08&4.6e-08&8.9e-08&3.1e-08&2.3e-08&6.6e-08\\
 		\hline
 			
 			\hline
 		\end{tabular}
 	\centering
 	\caption{\small{Numerical results for FGMRES to solve Example \ref{ex1} with $Q_{1}-P_{0} $ FEM on stretched grids.}}\label{tab2}\vspace{0.25cm}
 	\begin{tabular}{|p{0.8cm}| p{1.3cm}|p{1.4cm}p{1.4cm}p{1.4cm} p{1.5cm} p{1.8cm}p{1.9cm}| }
 		\hline
 		\multirow{1}{*}{Prec.}   &$ h (DOF) $           &  $\frac{1}{8}$ (286)   & $\frac{1}{16} (1086) $  &  $\frac{1}{32} (4222) $  &  $\frac{1}{64} (16638) $& $\frac{1}{128}(66046)$& $\frac{1}{256}(263166)$\\
 		\hline\hline
 		\multirow{3}{*}{$\mathcal{I}$}&
 		IT&48&133&93&187&406&974\\
 		& CPU&0.0391&0.0625&0.1089&0.5716&3.5570&118.1586\\
 		& RES&3.2e-08&7.5e-08&9.9e-08&9.9e-08&9.9e-08&1.0e-07\\
 		\hline
 		\multirow{4}{*}{$\mathcal{M}_{\alpha}$}&
 		$\alpha_{est}$&0.3962&0.2001&0.1011&0.0051&0.0025&0.0013\\
 		&IT&11&13&7&7&7&9\\
 		& CPU&0.0739&0.0846&0.0807&0.1851&0.5726&8.6066\\
 		& RES&6.4e-08&3.5e-08&3.1e-08&6.5e-08&9.0e-08&3.1e-08\\
 		\hline
 		
 		\hline
 	\end{tabular}
 \centering
 \caption{\small{Numerical results for FGMRES to solve Example \ref{ex1} with $Q_{2}-P_{1} $ FEM on uniform grids.}}\label{tab3}\vspace{0.25cm}
 \begin{tabular}{|p{0.8cm}| p{1.3cm}|p{1.4cm}p{1.4cm}p{1.4cm} p{1.5cm} p{1.8cm}p{1.9cm}|  }
 	\hline
 	\multirow{1}{*}{Prec.}   & $h (DOF) $           &  $\frac{1}{8}(254)$   & $\frac{1}{16} (958) $  &  $\frac{1}{32} (3710) $  &  $\frac{1}{64} (14590) $& $\frac{1}{128}(57854)$&$\frac{1}{256} (230398)$ \\
 	\hline\hline
 	\multirow{3}{*}{$\mathcal{I}$}&
    IT&51&130&186&253&947&$\dagger$\\
 	& CPU&0.0391&0.0653&0.1494&0.6904&7.8228&\\
 	& RES&8.8e-08&9.0e-08&9.8e-08&9.9e-08&1.0e-07&\\
 	\hline
 	\multirow{4}{*}{$\mathcal{M}_{\alpha}$}&
 	$\alpha_{est}$&0.0419&0.0214&0.0108&0.0054&0.0027&0.0013\\
 	&IT&11&14&13&10&12&12\\
 	& CPU&0.0737&0.0771&0.0978&0.2462&1.0375&9.0643\\
 	& RES&3.5e-08&2.6e-08&7.7e-08&7.3e-08&5.8e-08&9.0e-08\\
 	\hline
 	
 	\hline
 \end{tabular}
\centering
\caption{\small{Numerical results for FGMRES to solve Example \ref{ex1} with $Q_{2}-P_{1} $ FEM on stretched grids.}}\label{tab4}\vspace{0.25cm}
\begin{tabular}{|p{0.8cm}| p{1.3cm}|p{1.4cm}p{1.4cm}p{1.4cm} p{1.5cm} p{1.8cm}p{1.9cm}| }
	\hline
	\multirow{1}{*}{Prec.}   & $h (DOF) $           &  $\frac{1}{8}(254)$   & $\frac{1}{16} (958) $  &  $\frac{1}{32} (3710) $  &  $\frac{1}{64} (14590) $& $\frac{1}{128}(57854)$&$\frac{1}{256} (230398)$ \\
	\hline\hline
	\multirow{3}{*}{$\mathcal{I}$}&
	IT&55&131&121&198&512&1137\\
	& CPU&0.0380&0.0782&0.1123&0.5466&4.1606&129.2754\\
	& RES&6.5-08&1.0e-07&1.0e-07&9.5e-08&9.9e-08&1.0e-07\\
	\hline
	\multirow{4}{*}{$\mathcal{M}_{\alpha}$}&
	$\alpha_{est}$&0.4195&0.0214&0.0108&0.0054&0.027&0.014\\
	&IT&11&13&8&9&9&10\\
	& CPU&0.0732&0.0722&0.0833&0.2363&0.9544&10.7704\\
	& RES&9.6e-08&9.5e-08&9.8e-08&2.1e-08&4.2e-08&8.0e-08\\
	\hline
	
	\hline
	\end{tabular}
 \end{table}
 \end{example}
 \begin{example}\label{ex2} \rm
 	Consider a technical modification of Example 1 in \cite{Huang1,CAMWA} as follows:
 	\begin{align*}
 	A=\left(\begin{array}{cc}
 	{I \otimes T+T \otimes I} & {0} \\
 	{0} & {I \otimes T+T \otimes I}
 	\end{array}\right) \in \mathbb{R}^{2 p^{2} \times 2 p^{2}},
 	\end{align*}
 	$$
 	B=(I \otimes F \quad F \otimes I) \in \mathbb{R}^{p^{2} \times 2 p^{2}}$$
 	 and  
 	 
 	 $$C=\left(\begin{array}{l}
 	 C_{1} \\
 	 C_{2}
 	 \end{array}\right)=\left(\begin{array}{l}
 	 C_{1} \\
 	 c_{2}\\
 	 c_{2}
 	 \end{array}\right) \in \mathbb{R}^{(p^{2}+2)\times p^{2}}, $$
 	 where $$C_{1}=
 	 E \otimes F \in \mathbb{R}^{p^{2} \times p^{2}},$$ 
 	 and
 	 $$c_1=
 	 \left(\begin{array}{ll}
 	 e^{T}; &0^{T} \end{array}\right) C_{1},
 	 \qquad c_2=
 	 \left(\begin{array}{ll}
 	 0^{T}; &e^{T} \end{array}\right) C_{1},
 	 \qquad e^{T}=	\left(\begin{array}{llll}
 	 1; &1;&1;\cdots,&1 \end{array}\right) \in \mathbb{R}^{\frac{p^{2}}{2}}. $$
 	where
 	\begin{align*}
 	T=\frac{1}{h^{2}} \cdot \operatorname{tridiag}(-1,2,-1) \in \mathbb{R}^{p \times p}, \quad F=\frac{1}{h} \cdot \operatorname{tridiag}(0,1,-1) \in \mathbb{R}^{p \times p},
 	\end{align*}
 	and $
 	E=\operatorname{diag}\left(1, p+1, 2p+1, \ldots, p^{2}-p+1\right)
 	$ in which	the Kronecker product is denoted by $\otimes$, while the discretization mesh size is represented by $h={1}/{(p+1)}$.\\
 	
 	Table \ref{tab5} reports the result of the FGMRES method and the FGMRES in conjunction with the proposed preconditioner $\mathcal{M}_{\alpha},$ with respect to IT, CPU and RES. Table \ref{tab5} shows that the suggested preconditioner requires significantly less iteration numbers and CPU time than the FGMRES without a preconditioner. We also see that when the size of the problem is increased, the number of iterations of the FGMRES without preconditioning increase drastically, whereas for the FGMRES with the preconditioner $\mathcal{M}_{\alpha}$ increases moderately.  Figure \ref{fig2} illustrates that the preconditioned matrix works well in clustering the eigenvalues.
  	 \begin{figure}[H]
 	\begin{center}
 		\includegraphics[height=3.75cm,width=5cm]{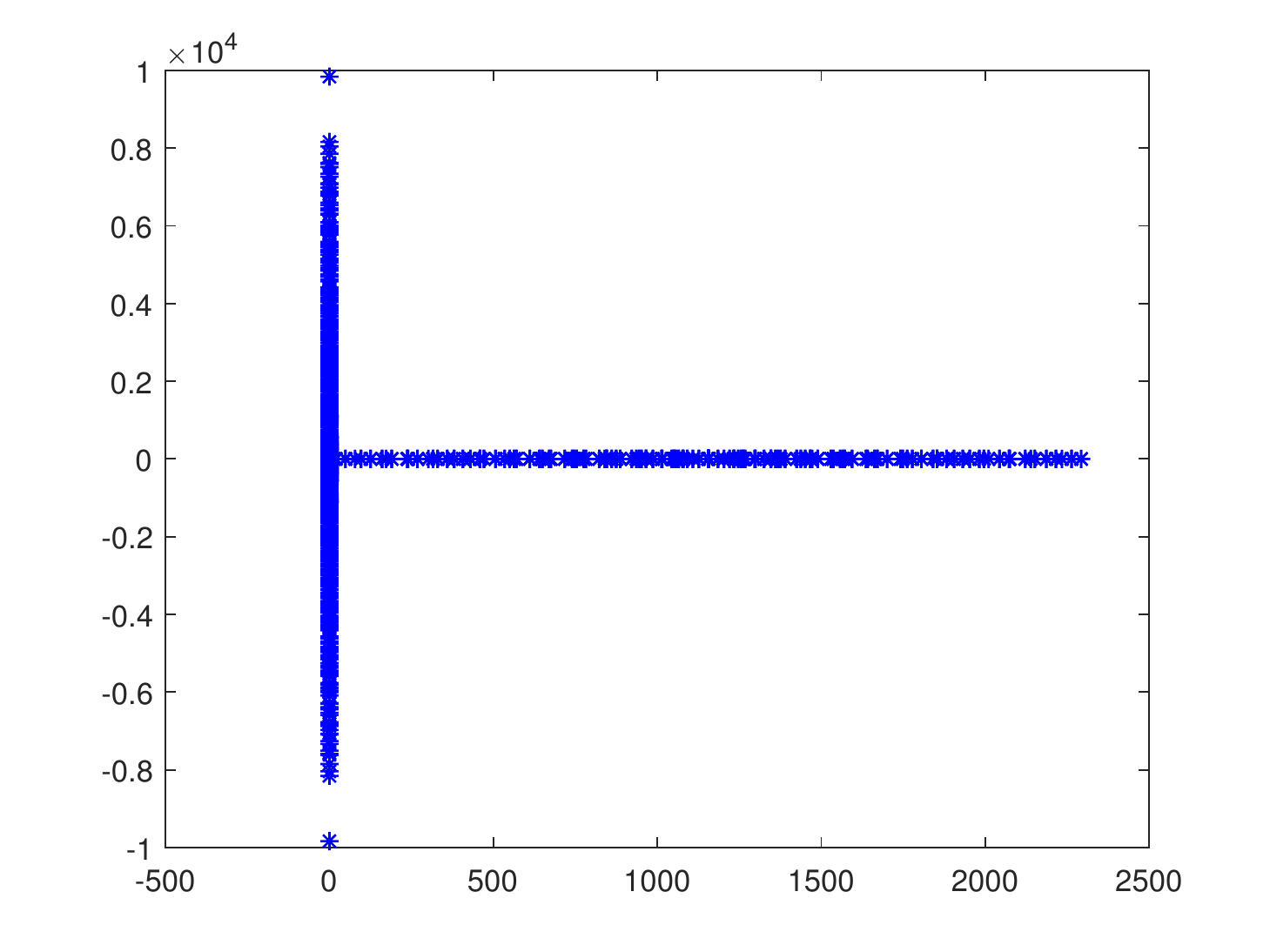}
 		\hspace*{0.2cm}		
 		\includegraphics[height=3.75cm,width=5cm]{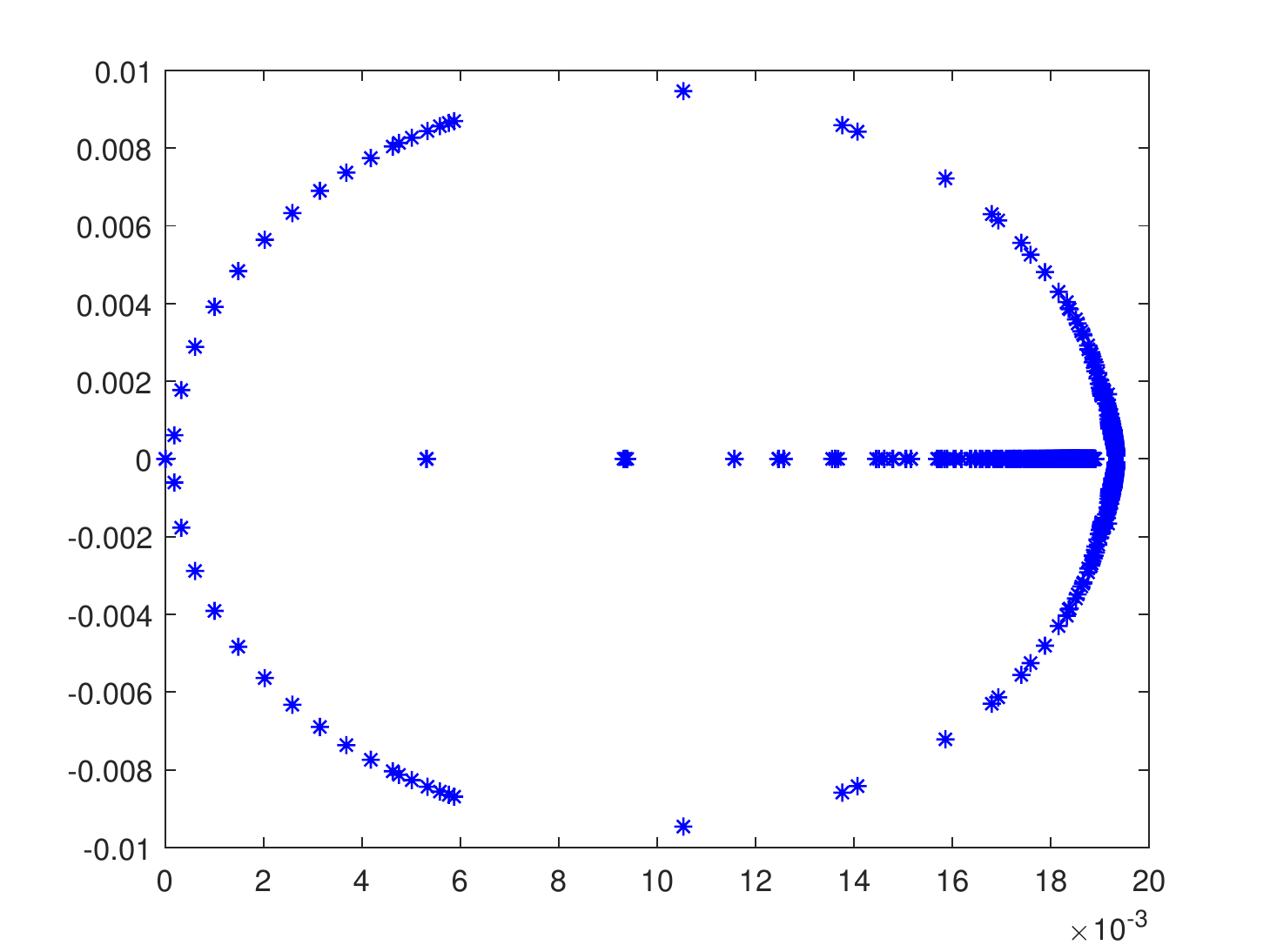}
 		\hspace*{0.2cm}
 		
 	\end{center}
 	\caption{Eigenvalue distributions of $\mathcal{A}$ and $\mathcal{M}_{\alpha}^{-1} \mathcal{A}$ (from the left to right) for $h=\frac{1}{16}$ for Example \ref{ex2} .} \label{fig2}
 \end{figure}
 		\begin{table}[!t]
 		\centering
 		\caption{Numerical results for FGMRES to solve Example \ref{ex2} .}\label{tab5}\vspace{0.25cm}
 		\begin{tabular}{|p{0.8cm}| p{1.3cm}|p{1.3cm}p{1.7cm}p{1.7cm} p{1.8cm} p{2cm}| }
 			\hline
 			\multirow{1}{*}{Prec.}   & $p (DOF)$            &  8 (258)   & 16 (1026)   &  32 (4098)   &  64 (16386) & 128 (65538) \\
 			\hline\hline
 			\multirow{4}{*}{$\mathcal{I}$}&
 			
 			IT&659&1999&$\dagger$&$\dagger$&$\dagger$\\
 			& CPU&0.1371&0.5270&&&\\
 			& RES&9.9e-08&1.0e-07&&&\\
 			\hline
 			\multirow{4}{*}{$\mathcal{M}_{\alpha}$}&
 			$\alpha_{est}$&0.0434&0.0219&0.0110&0.0055&0.0027\\
 			&IT&13&14&15&17&27\\
 			& CPU&0.0700&0.0857&0.1448&0.5156&2.5788\\
 			& RES&2.4e-08&3.6e-08&4.5e-08&8.3e-08&7.3e-08\\
 			\hline
 			
 			\hline 		
 		\end{tabular}
 	\end{table}
 	
 	\end{example}
 \section{Conclusions} \label{Sec6}
 In this paper, the APSS method was employed to solve a class of nonsymmetric three-by-three singular saddle point problems. We have applied the induced preconditioner, $\mathcal{M}_{\alpha},$ to improve the semi-convergence  rate ,when it is conjucted with FGMRES. We have proved that if $C$ is rank-deficient the APSS method is unconditionally semi-convergent. Numerical tests prove our theoretical claims.
 
 \section{Declaration}
 There is no conflict of interest.
 
 \section{Acknowledgements}
 The authors would like to that Prof. Fatemeh Panjeh Ali Beik (Vali-e-Asr University of Rafsanjan) and Dr. Mohsen Masoudi (University of Guilan) for their useful comments on an earlier version of this paper.

\end{document}